\documentclass[12pt,twoside]{amsart}
\usepackage{amsmath, amssymb,latexsym}
\usepackage[all,cmtip]{xy}
\usepackage{url}
\usepackage{amsmath,amsthm}
\usepackage{enumerate}
\usepackage{graphicx}

\theoremstyle{theorem}
\newtheorem{theorem}{Theorem}
\newtheorem{lemma}[theorem]{Lemma}

\newtheorem{proposition}[theorem]{Proposition}

\theoremstyle{definition}
\newtheorem{definition}[theorem]{Definition}

\theoremstyle{remark}
\newtheorem{remark}[theorem]{Remark}

\newcommand{\dd}
{\displaystyle}

\begin{document}
\title[First Order Algebras]{The Universal Theory of First Order Algebras and Various Reducts}
\author{Lawrence Valby}
\address{valby@berkeley.edu}

\maketitle

\begin{abstract}

First order formulas in a relational signature can be considered as operations on the relations of an underlying set, giving rise to multisorted algebras we call first order algebras. We present universal axioms so that an algebra satisfies the axioms iff it embeds into a first order algebra. Importantly, our argument is modular and also works for, e.g., the positive existential algebras (where we restrict attention to the positive existential formulas) and the quantifier-free algebras. We also explain the relationship to theories, and indicate how to add in function symbols.
   
\end{abstract}

\section{Introduction}

Briefly speaking, we present in this paper an axiomatization of the universal theory of certain classes of multisorted algebras arising from intersection, union, and other first order operations on relations.\footnote{By ``algebra" I mean a structure in a signature with only function symbols (including constants), i.e.\ a functional signature.}  A reasonable axiomatization of the Horn clause theory of these classes was already established by B\"orner \cite{Boerner}. Theorems \ref{main theorem}, \ref{negation theorem}, \ref{main theorem positive existential}, and \ref{main theorem first order} establish an axiomatization of the universal theory for various situations depending on which first order operations are included in the signature. In Section~\ref{positive quantifier-free} we introduce an axiom (axiom (0)) which spells the difference between the Horn clause theory and the universal theory in each situation. 

Speaking more generally, we investigate in this paper logical connectives like ``and" ($\wedge$), ``or" ($\vee$), ``not" ($\neg$), and ``there exists" ($\exists$) in an algebraic way. We regard logical connectives as operations on relations. For example ``and" corresponds to intersection of relations. Our goal is to understand better how these operations relate to each other. An example of one well-known fact about the connectives is that $\neg(R\wedge S)=\neg R\vee\neg S$. But many other facts are true too, and so we are in fact interested in finding some axioms from which all such facts follow. This was done in the case of the propositional connectives (and, or, not) by Stone in 1936 \cite{Stone}. Our topic of first order connectives has been studied too. Our formalism of choice for discussing this topic is a multisorted one. We will be dealing with first order algebras (and reducts), to be defined in Section~\ref{preliminaries}, which are multisorted algebras of relations with the possible arities of the relations being the sorts. Schwartz \cite{Schwartz} and B\"orner \cite{Boerner} have studied these structures, but they focused on the Horn clause theory of them: we shall be axiomatizing the universal theory.\footnote{I note that the Horn clause theory and the equational theory of first order algebras are equivalent, and similarly for the various reducts considered in this paper. See, e.g., Proposition~\ref{Horn clause}.} 
 
Cylindric algebras provide another formalism for investigating first order operations on relations. The cylindric set algebras, say of dimension $\omega$, (notated $\operatorname{Cs}_\omega$ in Henkin, Monk, and Tarski's book \textit{Cylindric Algebras} \cite{HMT}) in one sense correspond to the subalgebras of first order algebras. The class $SP(\operatorname{Cs}_\omega)=\operatorname{Gs}_\omega$ has an equational axiomatization. However, the class $\operatorname{Cs}_\omega$ itself is not first order axiomatizable (see Remark~3.1.100 in \cite{HMT}). On the other hand, in the multisorted situation, the subalgebras of the first order algebras are universally axiomatizable. The locally finite regular cylindric set algebras, in symbols $\operatorname{Cs}^{\operatorname{reg}}_\omega\cap \operatorname{Lf}_\omega$, likewise form a class which is not first order axiomatizable, yet perhaps they more closely correspond to our multisorted algebras. There is a characterization of their universal theory, due to Andr\'{e}ka and N\'{e}meti (see Theorem~4.1.48, p.\ 127 and 129 in \cite{HMT}). However, it is not immediately clear how to translate between the cylindric algebra and multisorted formalisms, and the axiomatizations and proofs seem different. Also, these results have all the first order operations present, while the argument here explicitly addresses various reducts as well.

Another approach, developed by Craig, is to examine sets of finite sequences of various lengths \cite{craig}. In one respect it is like our multisorted algebras because it allows for relations of various arities, but in another respect it is different because it too is single-sorted. Craig's work in \cite{craig} focuses on the equational theory and not the universal theory. One particular class of algebras, the finite sequence set algebras, is also described on p.\ 265 of \cite{HMT}, and its equational theory indeed has a reasonable equational axiomatization, but on the other hand it is not first order axiomatizable (for a similar reason to that given by Remark~3.1.100 in \cite{HMT}).

Let me stress that in this paper we are more interested in the universal theory than the Horn clause theory of first order algebras and various reducts. Indeed, a reasonable axiomatization of the Horn clause theory of first order algebras is an algebraic version of the completeness theorem for first order logic and has already been noted in the multisorted formalism by Schwartz \cite{Schwartz} and B\"orner \cite{Boerner} and previously in other formalisms (e.g., a cylindric algebra approach in \cite{HMT} or \cite{AndrekaNemeti}). B\"orner's Theorem 3.4.28 in \cite{Boerner} is essentially our Proposition~\ref{Horn for positive existential}. Similarly, B\"orner's Theorem 3.2.7 is our Propositions~\ref{Horn clause} and \ref{Horn negation}. We include these propositions about the Horn clause theory here for comparison and because they fall out very naturally from our approach to the universal theory. The new result of this paper is Theorem~\ref{main theorem}, as well as its straightforward generalization to other signatures: Theorems~\ref{negation theorem}, \ref{main theorem positive existential}, and \ref{main theorem first order}. This universal theory result is modular and works for various signatures. We do not specifically deal with each possible signature in detail, but rather deal with a few signatures in detail, and then briefly indicate how to modify the argument for other situations: Section~\ref{equality} deals briefly with equality and Section~\ref{function symbols} with function symbols.  

Models of the Horn clause theory of first order algebras are essentially first order theories (as discussed in Section~\ref{theories}). Thinking about first order theories in an algebraic way is not new, and it was discussed even in our multisorted formalism by B\"orner in Section~3.7 of \cite{Boerner}. There are a number of discussions of theories in the algebraic logic literature: see for example Section~4.3 of \cite{HMT}. Models of the universal theory of first order algebras are first order theories that are further complete, in the sense that every sentence is true or false and not both. However, the same is not true for the various reducts, e.g.\ positive existential algebras, and in this paper we provide a similarly intuitive condition that spells out the difference between the universal and Horn clause theories for various reducts uniformly (see axiom (0) discussed in Section~\ref{positive quantifier-free} below). We include Section~\ref{theories} on theories not because essentially new material is being introduced, but rather to help explain the value of axiom (0).

This paper is not a category-theoretical approach to thinking about Stone duality in the context of first order logic (cf.\ \cite{AwodeyForssell}), but rather a multisorted algebraic approach to axiomatizing certain concrete classes of algebras of operations on finitary relations. The kind of algebra appearing in this paper is called a ``many-sorted cylindric algebra'' in the survey \cite{Nemeti}. In B\"orner's work they are called ``Krasner algebras''. Bernays considered a similar kind of multisorted-like system, which is presented on pp.\ 263-264 of \cite{HMT}, but that focused on first order logic and not reducts.

Let us now motivate our subject by first considering the case of propositional logic. Consider the propositional formula $\varphi$ with two proposition letters $P$ and $Q$ given by $P\wedge\neg Q$. We can think of such a formula $\varphi$ as giving rise to an operation on subsets of a set. If $p$ and $q$ are subsets of some set $W$, then $\varphi(p,q):=p\cap(W-q)$ is also a subset of $W$. This function $\varphi\colon {\mathcal P}(W)\times {\mathcal P}(W)\to {\mathcal P}(W)$ accepts as input two subsets of $W$ and outputs a subset. Indeed, every propositional formula gives rise to a finitary operation ${\mathcal P}(W)^n\to {\mathcal P}(W)$. In this way we arrive at a functional signature $\tau$ where there is a function symbol of arity $n$ for every propositional formula involving $n$ proposition letters, and we have for every set $W$ a $\tau$-algebra whose underlying set is the powerset ${\mathcal P}(W)$ and so we may call it a \textit{powerset algebra}. Of course, we do not work in practice with the full signature $\tau$, but rather we isolate just some of the operations which compositionally generate all the others in the class of algebras of interest. One convenient choice is $0,1,\vee,\wedge,\neg$ where 0 and 1 are the constants interpreted by $\emptyset$ and $W$ in the powerset algebra determined by $W$. 

Stone's representation theorem gives equational axioms so that a $\tau$-algebra satisfies the axioms if and only if it embeds into a powerset algebra \cite{Stone}. This result can be understood in two different directions as it were. On the one hand, we can imagine being presented with some axioms (the Boolean algebra axioms in this case) and wishing to find some geometrical representation of any algebra that satisfies these axioms. On the other hand, we could begin with a geometric or otherwise natural class of algebras, and then try to axiomatize this class. It is this latter point of view which motivates the present paper.

If we use $K$ to denote the class of powerset algebras, the class $S(K)$ of subalgebras of powerset algebras consists of the $\tau$-algebras whose elements are \textit{some} of the subsets of some set with the operations interpreted normally. Stone's result yields an equational axiomatization of $S(K)$. A $\tau$-algebra $M$ is in $S(K)$ if and only if $M$ satisfies the Boolean algebra axioms.

Analogously, Cayley's theorem that any group embeds into the group of all permutations on some set can be read in two ways. On the one hand, we are showing that any algebra satisfying certain axioms (the group axioms) is representable in a certain way. On the other hand, we are finding an axiomatization of a certain natural class of algebras (the permutation groups). 

If instead of having an operation for every propositional formula we only have one for every positive propositional formula (in other words we look at the operations generated by $0,1,\vee,\wedge$), then a version of Stone's argument works to axiomatize the subalgebras of the powerset algebras in this reduced signature \cite{Stone2}. Once again, this result could be expressed as saying that every distributive lattice has a certain geometric representation, but for the purposes of this paper I prefer to think of it as a way of arriving at the distributive lattice axioms. 

Now instead of having an operation for every propositional formula, let us have an operation for every first order formula in a finite relational signature. For example, if $\theta$ is the formula $\exists y[R_1(x,y)\wedge R_2(x,y)]$, then as an operation on relations $\theta$ accepts as input two binary relations $r_1,r_2\subseteq W^2$ and outputs the unary relation $\{x\in W\mid \exists y[r_1(x,y)\wedge r_2(x,y)]\}$ which is a projection of their intersection. For the next example it is important to note that we consider a formula to come specifically equipped with a \textit{variable context} that contains the free variables of the formula but could contain variables otherwise not explicitly occuring in the formula. Let $\theta$ now be the formula $R(x,y)$ in the variable context $(x,y,z)$. Then as an operation $\theta$ accepts as input a binary relation $r\subseteq W^2$ and outputs the 3-ary relation $\{(x,y,z)\in W^3\mid (x,y)\in r\}$. We will be calling such an operation a \textit{cylindrification}. More generally, let $\sigma$ be the finite relational signature consisting of the relation symbols $R_1,\ldots,R_n$ of arities $m_1,\ldots,m_n$ respectively. Let $\theta(\bar{x})$ be a first order $\sigma$-formula in the variable context $x_1,\ldots,x_k$. Let $W$ be any set. Then $\theta$ induces a function on relations $\theta\colon {\mathcal P}(W^{m_1})\times\cdots\times {\mathcal P}(W^{m_n})\to {\mathcal P}(W^k)$ defined by
\[\theta(r_1,\ldots,r_n):=\{\bar{x}\in W^k\mid (W,r_1,\ldots,r_n)\models\theta(\bar{x})\}\]
where $(W,r_1,\ldots,r_n)$ denotes the $\sigma$-structure where each relation symbol $R_i$ is interpreted as $r_i$ and $\models$ denotes the usual notion of satisfaction. 

We see that the operations arising from first order formulas are a little bit different from those arising from propositional formulas in that now there are different sorts ${\mathcal P}(W^0)$, ${\mathcal P}(W^1)$, ${\mathcal P}(W^2)$, etc. Instead of a single-sorted algebraic signature, the first order formulas naturally give rise to a multisorted algebraic signature where there is a sort for each natural number. Every set $W$ gives rise to a multisorted algebra in this signature, with the operations as defined above. We call the algebras that arise in this way \textit{first order algebras}. 

In this paper we present a universal axiomatization of the subalgebras of first order algebras. The argument we use importantly also works when dealing with various reducts. For the classes of algebras $K$ of interest to us here, a universal axiomatization of the subalgebras of $K$ is the same thing as a universal axiomatization of the universal theory of $K$, hence the title of this paper is appropriate. It is not possible to replace our universal axiomatization with an equational one, nor even a Horn clause one, because the algebras in question are not closed under products. 

I would like to thank Alex Kruckman, Sridhar Ramesh, and Tom Scanlon for important input.

\section{Preliminaries}
\label{preliminaries}

Instead of dealing with a signature where there is a function symbol for every first order formula, it suffices to deal with a subsignature which will compositionally generate all the operations of interest. There is of course some degree of choice here, and we have generally chosen so as to make our axioms and arguments to follow more conveniently stated. Below is the (largest) signature we will use. Note the convention that ``$x\colon A$'' indicates $x$ is a constant of sort $A$, and ``$x\colon A\to B$'' indicates that $x$ is a function symbol with domain $A$ and codomain $B$ --- in this case $A$ may be a sort or a product of sorts.
\begin{definition} The multisorted \textbf{signature of first order algebras} (with equality) is given as follows.
\label{signature}
\begin{itemize}
\item We have a sort $n$ for each natural number $n\in\{0,1,2,\ldots\}$. The sort $n$ is intended to consist of $n$-ary relations on some set.
\item For each function $\alpha\colon\{1,\ldots,n\}\to\{1,\ldots,k\}$ we have a function symbol $\alpha\colon n\to k$. These are called \textbf{substitutions} and will correspond to the operations arising from atomic formulas.
\item For each $n$ we have a constant symbol $0^n$ belonging to sort $n$, which we may write as $0^n\colon n$. Likewise we have constant symbols $1^n\colon n$ and function symbols $\vee^n\colon n\times n\to n$, $\wedge^n\colon n\times n\to n$, and $\neg^n\colon n\to n$. We usually omit the superscript and write simply $0,1,\vee,\wedge,\neg$, leaving the arity implicit to the context. 
\item For each $n$ we have a function symbol $\exists^n\colon n+1\to n$, which we will generally write $\exists$. This will correspond to projection or existential quantification of the last coordinate. 
\item For each $n$ with $1\leq i,j\leq n$ we have a constant $\Delta_{i,j}^n\colon n$. These will correspond to equality of various coordinates.
\end{itemize}
\end{definition}

Before going further, let us introduce some notation involving the substitutions. A function $\alpha\colon\{1,\ldots,n\}\to\{1,\ldots,k\}$ can also be described as a sequence of length $n$ with repetition allowed taken from a $k$-element set. Thus, $\alpha$ gives a way of transforming any $k$-tuple into an $n$-tuple. Let $W$ be a set. Define $\alpha^{\text{tuple}}\colon W^k\to W^n$ to be the obvious function induced on tuples. In detail, $\alpha^{\text{tuple}}(x_1,\ldots,x_k):=(x_{\alpha(1)},\ldots,x_{\alpha(n)})$. This function on tuples in turn induces a function on relations of particular interest, the inverse image. I.e., define $\alpha^{\text{relation}}\colon {\mathcal P}(W^n)\to {\mathcal P}(W^k)$ by $\alpha^{\text{relation}}(r):=\{\bar{x}\mid \alpha^{\text{tuple}}(\bar{x})\in r\}$. An atomic formula like $R(x,x,y,x)$ in the variable context $(x,y,z)$ corresponds to the substitution $\alpha^{\text{tuple}}(x,y,z)=(x,x,y,x)$. Both give rise to the same operation on relations which accepts as input a 4-ary relation $r$ and outputs a 3-ary relation $\{(x,y,z)\mid (x,x,y,x)\in r\}$. We generally use lower case Greek letters near the beginning of the alphabet to denote substitutions, e.g.\ $\alpha,\beta,\gamma$.  

\begin{definition}
We say a \textbf{cylindrification} is a substitution where the function $\alpha\colon\{1,\ldots,n\}\to\{1,\ldots,k\}$ is increasing. In other words, as a tuple of symbols $\alpha^{\text{tuple}}(x_1,\ldots,x_k)$ is a subtuple of $(x_1,\ldots,x_k)$. We often use lowercase $c$ to denote cylindrifications. A collection $c_1,\ldots,c_m$ of cylindrifications are called \textbf{partitioning cylindrifications} when they take the form $c_i^{\text{tuple}}(\bar{x}_1,\cdots,\bar{x}_m)=\bar{x}_i$. I.e., we have $c_i\colon k_i\to n$ where $n=k_1+\cdots+k_m$, and the function $c_i\colon\{1,\ldots,k_i\}\to\{1,\ldots,n\}$ is given by $c_i(l)=l+\sum_{j=1}^{i-1} k_j$. 
\end{definition}

Here is an example of some partitioning cylindrifications: $c_1^{\text{tuple}}(x,y)=x$ and $c_2^{\text{tuple}}(x,y)=y$, so that $c_1^{\text{relation}}(r)=\{(x,y)\mid x\in r\}$ and $c_2^{\text{relation}}(s)=\{(x,y)\mid y\in s\}$. 

We use $\operatorname{id}$ to denote the identity substitution $\operatorname{id}\colon n\to n$ for each $n$. Let $\alpha\colon n\to k$ and $\beta\colon k\to m$ be two composable substitutions. Note that $(\beta\circ\alpha)^{\text{relation}}=\beta^{\text{relation}}\circ\alpha^{\text{relation}}$. To be able to say this axiomatically, we need different notation for the two compositions. We shall use $(\beta\circ\alpha)(r)$ for the former and $\beta(\alpha(r))$ for the latter. 

Now we define the classes of algebras of interest to us (in our leaner signature).

\begin{definition}
A \textbf{first order algebra} (with equality) is an algebra in the multisorted signature specified in Defintion~\ref{signature} that arises from some set $W$ in the following way:
\begin{itemize}
\item The interpretation of the sort $n$ is ${\mathcal P}(W^n)$, the collection of all $n$-ary relations on $W$.
\item The interpretation of a substitution $\alpha\colon n\to k$ is \[\alpha^{\text{relation}}\colon {\mathcal P}(W^n)\to {\mathcal P}(W^k)\] as defined above.
\item The Boolean operations $0,1,\vee,\wedge,\neg$ are interpreted as usual on each sort.
\item Projection $\exists^n\colon n+1\to n$ is interpreted as expected. In detail, \[\exists^n(r):=\{(x_1,\ldots,x_n)\mid \exists y((x_1,\ldots,x_n,y)\in r)\}\] 
\item The constant $\Delta_{i,j}^n\colon n$ is interpreted as \[\Delta_{i,j}^n:=\{(x_1,\ldots,x_n)\mid x_i=x_j\}\]
\end{itemize}
\end{definition}
We will have occasion to look at various reducts of our signature, and the corresponding reducts of the first order algebras are given appropriate names. E.g., the \textbf{positive existential algebras} (without equality) are the reducts of the first order algebras to the signature not containing negation (for any sort), and not containing the constants for equality, but otherwise containing all the symbols. The \textbf{positive quantifier-free algebras} (without equality) are when we restrict attention to just the substitutions and the lattice operations $0,1,\vee,\wedge$ for each sort. 

Just as all first order formulas can be constructed from the atomic formulas using the Boolean connectives and existential quantification, so too is every operation on relations arising from a first order formula equivalent to a term in our signature when looking at the first order algebras. Similarly, there are terms for every positive existential formula in the positive existential algebras, etc. 

In the paper we freely use the Boolean prime ideal theorem in various forms: We use the compactness theorem for first order logic. We use the fact that there are prime filters extending filters in any distributive lattice. Similarly, when we have a filter disjoint from an ideal in a distributive lattice, we may introduce a prime filter extending the filter and not containing anything from the ideal.\footnote{These are known to be equivalent (modulo ZF say) \cite{Henkin}. This property is implied by the axiom of choice, but is known to be strictly weaker \cite{HalpernLevy}. It is also known to be independent of ZF \cite{Feferman}.}  

\section{Positive Quantifier-free Algebras}
\label{positive quantifier-free}

The core of our argument about the first order algebras and various reducts can already be illustrated with the positive quantifier-free algebras, where our signature is restricted to the substitutions and the lattice operations $0,1,\vee,\wedge$ for each sort. The kind of operations on relations you can get in this situation is limited; for example, you can't express composition of binary relations. We begin by presenting the universal axioms which we will see axiomatize the subalgebras of the positive quantifier-free algebras --- this is our goal in this section. Note that I have placed a list of all the axioms considered in this paper (for the various reducts) at the end of the paper for ease of reference. Also note that each of these ``axioms" is actually an axiom schema. 

\begin{center}
\textbf{The Positive Quantifier-free Axioms}
\end{center}
\begin{enumerate}[(1)]
\setcounter{enumi}{-1}
\item When $c_1,\ldots,c_m$ are partitioning cylindrifications of arities $c_i\colon k_i\to (k_1+\cdots+k_m)$ we have the axiom: For all $r_1,s_1\colon k_1$, and all $r_2,s_2\colon k_2$, $\ldots$, and all $r_m,s_m\colon k_m$ we have  
\begin{center}
If $\dd\bigvee_{i=1}^m c_i(s_i)\geq\bigwedge_{i=1}^m c_i(r_i)$, then $s_i\geq r_i$ for some $i=1,\ldots,m$.
\end{center}
Note that by $y\geq x$ we actually mean $x=x\wedge y$ or equivalently (in the presence of the next axiom) $y=x\vee y$.
\item $0,1,\vee,\wedge$ form a (bounded) distributive lattice in each sort. In particular, each sort comes with a partial order $\leq$ defined by $r\leq s$ just in case $r=r\wedge s$ or equivalently $s=r\vee s$.
\item Substitutions preserve $0,1,\vee,\wedge$. E.g., when $\alpha\colon k\to n$ we write: For all $r,s\colon k$ we have $\alpha(r\wedge^k s)=\alpha(r)\wedge^n\alpha(s)$.
\item When $\alpha\colon k\to n$ and $\beta\colon n\to m$ we have the axiom: For all $r\colon k$ we have \begin{center}$(\beta\circ\alpha)(r)=\beta(\alpha(r))$.\end{center} 
Recall that by ``$\beta\circ\alpha$" we mean the function symbol which is the composition of these two substitution function symbols, while $\beta(\alpha(\bullet))$ is the usual composition within the algebra.
\item For each identity substitution $\operatorname{id}\colon n\to n$ we have the axiom: For all $r\colon n$ we have 
\begin{center}
$\operatorname{id}(r)=r$.
\end{center}
\end{enumerate}

\begin{remark}
Here are some notes on the axioms, and intuitive explanations. 
\begin{itemize}
\item Intuitively, axiom (0) says that if you have a union of ``orthogonal cylinders'' covering a ``rectangle'', then one of the cylinders has width larger than the width of the corresponding side of the rectangle. Note that the instances of axiom (0) are not Horn clauses (they are universal implications where the conclusion is a disjunction of atomic formulas).

\item Axioms (1)-(4), which are equations, axiomatize the Horn clause theory of positive quantifier-free algebras. When axiomatizing the algebras of larger signatures, we will see that the difference between the universal and the Horn clause theories is still just axiom (0). (E.g., compare Theorem~\ref{main theorem} and Proposition~\ref{Horn clause}.)

\item Axiom (4) is redundant in the context of axioms (0), (1), and (3). However, I include it because we will have occasion to omit axiom (0) when considering the Horn clause theory. To see this redundancy, note that $\operatorname{id}\colon n\to n$ just by itself is trivially a partitioning cylindrification. Thus by axiom (0) we have $\operatorname{id}(s)\geq \operatorname{id}(r)$ implies $s\geq r$. But axiom (3) gives $\operatorname{id}(\operatorname{id}(t))=\operatorname{id}(t)$, and so we can get $\operatorname{id}(t)=t$. 

\item It is straightforward to check that the axioms are all true in positive quantifier-free algebras. For illustration, let us verify axiom (0). Suppose that $s_i\not\geq r_i$ for each $i=1,\ldots,m$. Then there are $\bar{x}_i\in r_i-s_i$, and so $(\bar{x}_1\cdots\bar{x}_m)\in\bigwedge_i c_i(r_i)-\bigvee_i c_i(s_i)$.

\item It follows from axioms (0), (1), and (3) that everything in sort zero is either 0 or 1. To see this, note that $c_1,c_2\colon 0\to 0$ are partitioning cylindrifications, where $c_1=c_2=\operatorname{id}$. Then given any element $r$ of sort zero, we have 
\[c_1(r)\vee c_2(0)=r\vee 0\geq 1\wedge r=c_1(1)\wedge c_2(r)\]
So either $r\geq 1$ (and hence $r=1$) or $0\geq r$ (and hence $r=0$).

\item We may consider $0\not\geq 1$ in sort zero to be a special case of axiom (0), because the empty collection of cylindrifications trivially forms a partitioning cylindrification (of sort zero). The right hand side of the axiom in this case becomes an empty disjunction and therefore is considered as FALSE. If this offends the reader's sensibilities, then they may specifically add an axiom asserting that $0\not\geq 1$ in sort zero. Taken together with the previous remark, we see that an algebra satisfying axioms (0), (1), and (3) will have exactly two elements in sort zero.    
\end{itemize}
\end{remark}

The main result of this paper will be the following theorem.

\begin{theorem}
\label{main theorem}
Axioms (0)-(4) axiomatize the subalgebras of the positive quantifier-free algebras. 
\end{theorem}

A basic step in our proof of Theorem~\ref{main theorem} will be the observation that if $L$ is an abstract algebra that satisfies the axioms above then a prime filter on any one of the sorts of $L$ gives rise to a morphism from $L$ to a concrete positive quantifier-free algebra.\footnote{Axiom (1) ensures that each sort is a distributive lattice, and so it makes sense to speak of a prime filter on a sort.}  If you think of the abstract algebra as a theory, and the morphism to a concrete algebra as a model of this theory, then intuitively Lemma~\ref{morphism} says that any prime filter $F$ on sort $n$ is the type of an $n$-tuple $\bar{F}=(F_1,\ldots,F_n)$ in some model. In fact, we can take the model to just consist of this $n$-tuple.
Lemma~\ref{sum} will allow us to realize finitely many types at once, and so then by the compactness theorem we will be able to realize all types at once, yielding an embedding.

\begin{lemma}
\label{morphism}
Let $F$ be a prime filter on sort $n$ of some algebra $L$ satisfying the axioms (1), (2), and (3). Let $F_1,\ldots,F_n$ be distinct symbols. Let $W=\{F_1,\ldots,F_n\}$. Let $A(W)$ denote the positive quantifier-free algebra of relations on the set $W$. Define a function (on each sort) $\varphi\colon L\to A(W)$ by putting, for each $r$ in sort $k$ of $L$ and each substitution $\alpha\colon k\to n$,
\[\varphi(r)=\{\alpha^{\operatorname{tuple}}(F_1,\ldots,F_n)\mid \alpha(r)\in F\}\]
Then $\varphi$ is a morphism.
\end{lemma}
\begin{proof}
Note that 
\[\alpha^{\operatorname{tuple}}(F_1,\ldots,F_n)\in\varphi(r)\iff \alpha(r)\in F\]
because every tuple (of any length) from $W$ can be expressed as $\alpha^{\text{tuple}}(\bar{F})$ for a unique substitution $\alpha$.  We now proceed to check that $\varphi$ is a morphism. First observe that $\varphi$ preserves 0, i.e.\ $\varphi(0)=\emptyset$ for each sort, because $\alpha(0)=0\not\in F$ by axiom (2). Similarly, $\varphi(1)=W^k$ because $\alpha(1)=1\in F$. 

The preservation of $\vee$ and $\wedge$ also follow from axiom (2) via the following calculations:
\begin{align*}
\alpha^{\text{tuple}}(\bar{F})\in\varphi(r)\cup\varphi(s)&\iff\alpha(r)\in F\text{ or }\alpha(s)\in F\\
&\iff \alpha(r\vee s)\in F\\
&\iff \alpha^{\text{tuple}}(\bar{F})\in\varphi(r\vee s)
\end{align*}
and
\begin{align*}
\alpha^{\text{tuple}}(\bar{F})\in\varphi(r)\cap\varphi(s)&\iff\alpha(r)\in F\text{ and }\alpha(s)\in F\\
&\iff \alpha(r\wedge s)\in F\\
&\iff\alpha^{\text{tuple}}(\bar{F})\in\varphi(r\wedge s)
\end{align*}

Finally, we check the preservation of substitutions using axiom (3):
\begin{align*}
\alpha^{\text{tuple}}(\bar{F})\in\beta^{\text{relation}}(\varphi(r))&\iff\beta^{\text{tuple}}(\alpha^{\text{tuple}}(\bar{F}))\in\varphi(r)\\
&\iff(\alpha\circ\beta)^{\text{tuple}}(\bar{F})\in\varphi(r)\\
&\iff (\alpha\circ\beta)(r)\in F\\
&\iff \alpha(\beta(r))\in F\\
&\iff \alpha^{\text{tuple}}(\bar{F})\in\varphi(\beta(r))
\end{align*}
\end{proof}

\begin{proposition}[B\"orner]
\label{Horn clause}
Axioms (1)-(4) axiomatize the class of subalgebras of products of positive quantifier-free algebras. Thus, we have found an equational axiomatization of the Horn clause theory of positive quantifier-free algebras. 
\end{proposition}
\begin{proof}
Let $L$ be an algebra satisfying axioms (1)-(4). We want to find an embedding of $L$ into a product of positive quantifier-free algebras.
Using axiom (4), we can make sure to separate any two distinct points using a morphism from Lemma~\ref{morphism}: Let $r\neq s$ in the same sort. Then there is a prime filter $F$ containing, say, $r$ and not $s$.\footnote{Recall we have taken the Boolean prime ideal theorem as an assumption.} Let $\varphi$ be the morphism obtained from Lemma~\ref{morphism}. Then $\bar{F}=\operatorname{id}(\bar{F})\in\varphi(r)-\varphi(s)$ because $\operatorname{id}(r)=r\in F$ and $\operatorname{id}(s)=s\not\in F$, by axiom (4). Taking a product of a bunch of such morphisms, we actually get an embedding of an algebra satsifying axioms (1)-(4) into a product of positive quantifier-free algebras. 
\end{proof}

Observe that this does not automatically give us the universal theory, because the positive quantifier-free algebras are not closed under products (even just look at the zero sort and observe that there must be exactly two elements in it). Note that while ${\mathcal P}(\biguplus W_i)=\prod {\mathcal P}(W_i)$ it is not the case that ${\mathcal P}((\biguplus W_i)^n)=\prod {\mathcal P}(W_i^n)$.

In order to have an embedding into an actual positive quantifier-free algebra instead of a product of them, we show that we can deal with all the prime filters at once by showing a certain first order theory is satisfiable. 

\vspace{0.5 cm} 

\noindent\textit{First Part of Proof of Theorem~\ref{main theorem}.}
Given an algebra $L$ that satisfies axioms (0)-(4), let us introduce a relation symbol $r$ of arity $n$ for each element $r$ of $L$ of each sort $n$. We also introduce constants $F_1,\ldots,F_n$ for each prime filter $F$ of $L$ on sort $n$. Let $T$ be the first order theory in this language with the following axiom schemata:
\begin{enumerate}[(A)]
\item $r(\bar{F})$ when $r\in F$ and $\neg r(\bar{F})$ when $r\not\in F$
\item The \textbf{morphic conditions}, i.e.
\begin{enumerate}[(i)]
\item $\forall\bar{x}\,\,\, \neg 0(\bar{x})$. We have such a sentence for the 0 of each sort.
\item $\forall\bar{x}\,\,\, 1(\bar{x})$
\item $\forall\bar{x}\,\,\, (r\vee s)(\bar{x})\iff (r(\bar{x})\text{ or } s(\bar{x}))$. Note that in $(r\vee s)(\bar{x})$ the ``$\vee$" is an operation of the algebra, while in $(r(\bar{x})\text{ or }s(\bar{x}))$ the ``or" is a logical symbol of the ambient first order logic. We have such a sentence for every pair $(r,s)$ of elements from the same sort.
\item $\forall\bar{x}\,\,\, (r\wedge s)(\bar{x})\iff (r(\bar{x})\text{ and } s(\bar{x}))$
\item $\forall\bar{x}\,\,\, (\alpha(r))(\bar{x})\iff r(\alpha^{\operatorname{tuple}}(\bar{x}))$. We have such a sentence for every substitution $\alpha\colon k\to n$ and element $r$ in $L$ of sort $k$.
\end{enumerate}
\end{enumerate}

It is straightforward to verify that a model of the morphic conditions of $T$ is (essentially) the same thing as a morphism from the algebra $L$ to a positive quantifier-free algebra. Item (A) of $T$ ensures that this morphism is 1-1 (on each sort). To show that this theory is satisfiable, we show that the theory is finitely satisfiable and then use the compactness theorem. Thus, it suffices to find a model satisfying all of item (B) but the instances of item (A) involving only finitely many prime filters $F^1,\ldots,F^m$. (We use superscript here to avoid confusion of the prime filters with the constants associated to each of them.) 

\noindent\textit{End of First Part of Proof of Theorem~\ref{main theorem}}

\vspace{0.5 cm}

The key idea for how to proceed is to assemble these finitely many prime filters into one master prime filter on a larger sort. We formalize this in the following lemma.

\begin{lemma}
\label{sum}
Let $L$ be an algebra that satisfies axioms (0), (1), and (2). Let $k_1+\cdots+k_m=n$. Let $c_i\colon k_i\to n$ be partitioning cylindrifications. Let $F^i$ be prime filters on sort $k_i$ respectively. Then there is a prime filter $G$ on sort $n$ such that for each $i=1,\ldots,m$, for all $r$ in $L$ of sort $k_i$ we have $c_i(r)\in G$ if and only if $r\in F^i$.
\end{lemma}
\begin{proof}
Let $A$ be the distributive lattice which is sort $n$ of $L$. Define 
 \[G_F:=\{z\in A\mid z\geq\bigwedge_{i=1}^m c_i(r_i)\text{ for some } r_i\in F^i\}\]
and 
\[G_I:=\{z\in A\mid \bigvee_{i=1}^m c_i(s_i)\geq z\text{ for some }s_i\not\in F^i\}\]
We claim that $G_F$ is a filter, $G_I$ is an ideal, and they are disjoint. First we show they are disjoint. Suppose, to get a contradiction, that  $\bigvee_i c_i(s_i)\geq z\geq\bigwedge_i c_i(r_i)$ where $s_i\not\in F^i$ and $r_i\in F^i$. Then $s_i\geq r_i$ for some $i$ by axiom (0), implying that $s_i\in F^i$ because $F^i$ is upward-closed, but as noted $s_i\not\in F^i$, and we have a contradiction. 

Next, observe that $1\in G_F$ since $1\in F^i$ for each $i$ and $1\geq \bigwedge_i c_i(1)$. Similarly $0\in G_I$. It follows at once that $0\not\in G_F$ and $1\not\in G_I$ because $G_F$ and $G_I$ are disjoint. 

It is obvious from the definitions that $G_F$ is upward-closed and $G_I$ is downward closed. 

Finally, suppose $z,z'\in G_F$. Let $z\geq \bigwedge_i c_i(r_i)$ and $z'\geq \bigwedge_i c_i(r_i')$ where $r_i,r_i'\in F^i$. Then 
$z\wedge z'\geq\bigwedge_i c_i(r_i\wedge r_i')$ by axiom (2), and $r_i\wedge r_i'\in F^i$ for each $i$. So $z\wedge z'\in G_F$. The argument that $z,z'\in G_I$ implies $z\vee z'\in G_I$ is similar.

Because $G_F$ and $G_I$ are disjoint, there is a prime filter $G$ such that $G_F\subseteq G$ and $G\cap G_I=\emptyset$. This $G$ is a prime filter satisfying the desired property. If $r\in F^i$, then $c_i(r)= c_i(r)\wedge\bigwedge_{j\neq i}c_j(1)\in G_F$, and so $c_i(r)\in G$. If $r\not\in F^i$, then $c_i(r)=c_i(r)\vee\bigvee_{j\neq i}c_j(0)\in G_I$, and so $c_i(r)\not\in G$. 
\end{proof}  

\noindent\textit{Continuation of Proof of Theorem~\ref{main theorem}}. Armed with this lemma, we may return to showing that the theory $T$ is finitely satisfiable. Given our finitely many prime filters $F^1,\ldots,F^m$, there is by Lemma~\ref{sum} a prime filter $G$ such that 
\[c_i(r)\in G\iff r\in F^i\]
for each $i$. Introduce distinct symbols $G_1,\ldots,G_n$, where $n=k_1+\cdots+k_m$, the sum of the arities of the $F^i$. Let $W=\{G_1,\ldots,G_n\}$. We will interpret the constants corresponding to each prime filter $F^i$ by $c_i^{\operatorname{tuple}}(\bar{G})$ respectively. By Lemma~\ref{morphism} we know that $\varphi\colon L\to A(W)$ defined by
\[\alpha^{\operatorname{tuple}}(\bar{G})\in\varphi(r)\iff \alpha(r)\in G\]
is a morphism. Further,
\[c_i^{\operatorname{tuple}}(\bar{G})\in\varphi(r)\iff c_i(r)\in G\iff r\in F^i\]
as desired. So $\varphi$ yields the desired model of the small portion of $T$ we gave ourselves. 

\noindent\textit{End of Proof of Theorem~\ref{main theorem}} 

\vspace{0.5 cm}

Unlike powerset algebras which have equivalent equational, Horn clause, and universal theories, the positive quantifier-free algebras (and other first order algebra reducts considered below) have only equivalent equational and Horn clause theories.

\section{Adding Negation, Projection, Equality}

\subsection{Negation}

It is relatively easy to extend the results of Section~\ref{positive quantifier-free} to algebras with negation, yielding an axiomatization of the subalgebras of the quantifier-free algebras. 

\begin{theorem}\label{negation theorem}
Axioms (0)-(6) axiomatize the subalgebras of the quantifier-free algebras. (Axioms (5) and (6) are given below.) 
\end{theorem}

Examining the argument of Section~\ref{positive quantifier-free}, the only place where there needs to be significant change is for Lemma~\ref{morphism}, where we need to now also verify that the function defined is morphic for negation. In other words, we want $\varphi(\neg r)=\neg\varphi(r)$. I.e., we want $\alpha(\neg r)\in F\iff\alpha(r)\not\in F$. One way to accomplish this is to add the following two equational axiom schemata to our list:

\begin{center}
\textbf{Axioms for Negation}
\end{center}
\begin{enumerate}[(1)]
\setcounter{enumi}{4}
\item When $\alpha\colon k\to n$ is a substitution we have the axiom: For all $r\colon k$ we have \begin{center}$\alpha(\neg r)=\neg\alpha(r)$.\end{center}
\item For each sort $n$, we have the axiom: For all $r\colon n$ we have \begin{center}$r\vee\neg r=1$ and $r\wedge\neg r=0$.\end{center}
\end{enumerate}

It is easy to check that these axioms are true in quantifier-free algebras.

\begin{lemma}
\label{morphism with negation}
Let $F$ be a prime filter in sort $n$ of some algebra $L$ satisfying the axioms (1)-(3), and (5)-(6). Let $F_1,\ldots,F_n$ be distinct symbols. Let $W=\{F_1,\ldots,F_n\}$. Let $A(W)$ denote the quantifier-free algebra of relations on the set $W$. Define a function (on each sort) $\varphi\colon L\to A(W)$ by putting
\[\alpha^{\operatorname{tuple}}(\bar{F})\in\varphi(r)\iff \alpha(r)\in F\]
Then $\varphi$ is a morphism.
\end{lemma}

\begin{proof}
The proof is the same as that for Lemma~\ref{morphism}, except now we must also check that $\varphi(\neg r)=\neg\varphi(r)$. Axiom (6) ensures that for any prime filter $F$, $\neg r\in F$ if and only if $r\not\in F$. Then using axiom (5) we have $\alpha(\neg r)\in F$ if and only if $\neg\alpha(r)\in F$ if and only if $\alpha(r)\not\in F$.
\end{proof}

As before we get the following proposition:

\begin{proposition}[B\"orner]
\label{Horn negation}
Axioms (1)-(6) axiomatize the class of subalgebras of products of quantifier-free algebras. Thus, we have found an equational axiomatization of the Horn clause theory of quantifier-free algebras.
\end{proposition}

\begin{proof}
 This follows from Lemma~\ref{morphism with negation} in the same way that Proposition~\ref{Horn clause} follows from Lemma~\ref{morphism}.
\end{proof}

\noindent\textit{Proof of Theorem~\ref{negation theorem}}.
The proof is the same as that for Theorem~\ref{main theorem}, except that the theory $T$ used in that proof changes in a very minor way: we must add preservation of negation to the morphic conditions. In detail, we add the following sentences to $T$:
\begin{enumerate}[(A)]
\setcounter{enumi}{1}
\item 
\begin{enumerate}[(i)]
\setcounter{enumii}{5}
\item $\forall\bar{x}\,\,\,(\neg r)(\bar{x})\iff \neg(r(\bar{x}))$
\end{enumerate}
\end{enumerate}
The addition of this to the theory $T$ does not change the rest of the argument because Lemma~\ref{morphism with negation} handles negation. Note that Lemma~\ref{sum} remains unchanged by an expansion of the signature.

\noindent\textit{End of Proof of Theorem~\ref{negation theorem}}

\subsection{Projection}

Adding projection takes more work than adding negation. As in the quantifier-free case, our argument below works whether negation is present or not, and so we will obtain the following two theorems. (The new axioms (7)-(10) are presented later in this section.)

\begin{theorem}\label{main theorem positive existential}
Axioms (0)-(4) and (7)-(10) axiomatize the subalgebras of the positive existential algebras. 
\end{theorem}

\begin{theorem}\label{main theorem first order}
Axioms (0)-(10) axiomatize the subalgebras of the first order algebras.
\end{theorem}

The proofs of these two theorems are essentially the same, and so for convenience we will focus attention on the positive existential case, i.e.\ Theorem~\ref{main theorem positive existential}.
The signature under consideration thus includes the substitutions, the lattice operations, and the projections, but not negation. Our general approach is to find a 1-1 function from the abstract algebra satisfying the axioms to a concrete positive existential algebra which is not quite a morphism because there are not enough witnesses, but then we modify the function to obtain an actual embedding by adding witnesses. 

Often when dealing with a projection $\exists\colon n+1\to n$ we wish to also speak of the \textbf{associated cylindrification} $c\colon n\to n+1$ given by $c(i)=i$, i.e.\ $c(\bar{x}y)=\bar{x}$. The operations $\exists$ and $c$ form a Galois connection, which is a special case of how direct image and inverse image form a Galois connection.  However, we present the situation equationally with the following axioms, which imply more than just this Galois connection.

\begin{center}
\textbf{Axioms for Projection}
\end{center}
\begin{enumerate}[(1)]
\setcounter{enumi}{6}
\item $\exists$ preserves 0 and $\vee$
\item For each projection $\exists\colon (n+1)\to n$ and associated cylindrification $c\colon n\to (n+1)$ we have the axiom: For all $r\colon (n+1)$ we have \begin{center}
$r\leq c(\exists(r))$
\end{center}
\item For each projection $\exists\colon (n+1)\to n$ and associated cylindrification $c\colon n\to (n+1)$ we have the axiom: For all $r\colon (n+1)$ and all $s\colon n$ we have \begin{center}
$\exists(r\wedge c(s))=\exists(r)\wedge s$. 
\end{center}
\item Let $\alpha_i\colon k_i\to m$ be substitutions for $i=1,\ldots, n$. Let $\beta_i\colon (k_i+1)\to (m+n)$ be the substitutions defined by $\beta_i^{\text{tuple}}(\bar{x}y_1\cdots y_n)=\alpha_i^{\text{tuple}}(\bar{x})y_i$. Then we have the axiom: For all $r_1\colon k_1+1$, $\ldots$, and all $r_n\colon k_n+1$ we have \[\exists^{(n)}(\bigwedge_{i=1}^n\beta_i(r_i))=\bigwedge_{i=1}^n\alpha_i(\exists(r_i))\]
where $\exists^{(n)}$ means we apply projection $n$ times.
\end{enumerate}

\begin{remark}
Here are some notes on the axioms for projection, and intuitive explanations.
\begin{itemize}
\item It is straightforward to check that these axioms are all true in the positive existential algebras. For illustration, consider axiom (10). Intuitively, this axiom says that casting an ensemble for a theatrical production involving $n$ roles is equivalent to finding a good actor for each role, as long as you do not care about how the team works together. A tuple $\bar{x}$ is in $\exists^{(n)}(\bigwedge_{i=1}^n\beta_i(r_i))$ if and only if there are $y_1,\ldots,y_n$ such that for each $i=1,\ldots,n$ we have $\alpha_i^{\text{tuple}}(\bar{x})y_i=\beta_i^{\text{tuple}}(\bar{x}\bar{y})\in r_i$. However, since each $y_i$ occurs on its own, this is equivalent to saying that for each $i=1,\ldots,n$ there is some $y_i$ with $\alpha_i^{\text{tuple}}(\bar{x})y_i\in r_i$, which is to say $\bar{x}$ is in $\bigwedge_{i=1}^n\alpha_i(\exists(r_i))$. 

\item Note that $r\leq s$ implies $\exists(r)\leq\exists(s)$ follows from axiom (7). Of course the substitutions are also increasing in this way because of axiom (2). 

\item It might appear that, by taking  $r=1$ in axiom (9), we could conclude that $\exists(c(s)) = s$.  This is usually correct, but not always. If we consider the algebra of relations on the empty set, and $s=1$ in sort zero, then $c(s)=1=0$ in sort one, and $\exists(c(s))=0\neq 1=s$. On the other hand, we do always have $\exists(c(s))\leq s$.   
\end{itemize}
\end{remark}

\begin{definition}
Let $L$ be an algebra in the positive existential signature, and let $A(W)$ be the positive existential algebra of relations on some set $W$. An \textbf{almost morphism} is a function (on each sort) $\varphi\colon L\to A(W)$ such that
\begin{enumerate}
\item $\varphi$ is morphic for the substitutions and the lattice operations
\item $\exists(\varphi(r))\subseteq\varphi(\exists(r))$
\end{enumerate}
I.e., an almost morphism is a morphism except for the possibility it might not satisfy $\exists(\varphi(r))\supseteq \varphi(\exists(r))$.
\end{definition} 

The modified version of Lemma~\ref{morphism} is as follows.

\begin{lemma}
\label{almost morphism}
Let $F$ be a prime filter in sort $n$ of some algebra $L$ satisfying the axioms (1), (2), (3), and (8). Let $F_1,\ldots,F_n$ be distinct symbols. Let $W=\{F_1,\ldots,F_n\}$. Let $A(W)$ denote the positive existential algebra of relations on the set $W$. Define a function (on each sort) $\varphi\colon L\to A(W)$ by putting
\[\alpha^{\operatorname{tuple}}(\bar{F})\in\varphi(r)\iff \alpha(r)\in F\]
Then $\varphi$ is an almost morphism.
\end{lemma}
\begin{proof}
The new thing we need to verify is that $\exists(\varphi(r))\subseteq\varphi(\exists(r))$. Let $\alpha^{\text{tuple}}(\bar{F})\in\exists(\varphi(r))$. There is a substitution $\beta$ such that $c^{\text{tuple}}(\beta^{\text{tuple}}(\bar{F}))=\alpha^{\text{tuple}}(\bar{F})$ and $\beta^{\text{tuple}}(\bar{F})\in\varphi(r)$. Thus, $\beta(r)\in F$. We want to check that $c^{\text{tuple}}(\beta^{\text{tuple}}(\bar{F}))\in\varphi(\exists(r))$, i.e.\ $(\beta\circ c)(\exists(r))\in F$. Well,
\begin{align*}
(\beta\circ c)(\exists(r))&=\beta(c(\exists(r)))\\
&\geq\beta(r)\\
&\in F
\end{align*}
\end{proof}

Lemma~\ref{almost morphism} does not immediately yield an axiomatization of the Horn clause theory, but rather the following lemma.

\begin{lemma}
\label{Horn separation}
Let $L$ be an algebra satisfying axioms (1)-(4) and (8). Let $r\neq s$ be two distinct elements in the same sort. Then there is an almost morphism $\varphi$ from $L$ to a positive existential algebra such that $\varphi(r)\neq\varphi(s)$.
\end{lemma} 

\begin{proof}
Follows from Lemma~\ref{almost morphism} in the same way that (a portion of) Proposition~\ref{Horn clause} follows from Lemma~\ref{morphism}.
\end{proof}

Similarly, the argument in the proof of Theorem~\ref{main theorem} applied to this situation does not immediately yield Theorem~\ref{main theorem positive existential}, but rather the following lemma.

\begin{lemma}
\label{1-1 almost morphism}
Let $L$ be an algebra satisfying axioms (0)-(4), and (8). Then there is a 1-1 almost morphism from $L$ to some positive existential algebra.
\end{lemma}

\begin{proof}
The proof is the same as that for Theorem~\ref{main theorem}, except that the morphic conditions of the theory $T$ become the \textbf{almost morphic conditions}. That is, instead of adding $\forall\bar{x}\,\,\,(\neg r)(\bar{x})\iff \neg(r(\bar{x}))$ as we did in the proof of Theorem~\ref{negation theorem}, we add
\begin{enumerate}[(A)]
\setcounter{enumi}{1}
\item
\begin{enumerate}[(i)]
\setcounter{enumii}{5}
\item $\forall\bar{x},y\,\,\,(r(\bar{x}y)\implies(\exists(r))(\bar{x}))$
\end{enumerate}
\end{enumerate}
\end{proof}

To go further, we need a way of turning an almost morphism into an actual morphism. The following lemmas help us accomplish this.

\begin{lemma}
\label{project}
Let $L$ be an algebra that satisfies the axioms (1), (2), (7), and (9). Let $F$ be some prime filter on sort $n$ and let $r$ be an element in sort $n+1$ such that $\exists(r)\in F$. Then there is some prime filter $G$ on sort $n+1$ such that $r\in G$ and for all $u$ in sort $n$ we have $c(u)\in G$ if and only if $u\in F$.  
\end{lemma} 
\begin{proof}
We use the same approach as in the proof of Lemma~\ref{sum}. Let $A$ be the distributive lattice which is sort $n+1$ of $L$. Define
\[G_F:=\{z\in A\mid z\geq r\wedge c(u)\text{ for some }u\in F\}\]
and 
\[G_I:=\{z\in A\mid c(u)\geq z\text{ for some }u\not\in F\}\]
Then as in Lemma~\ref{sum} we have that $G_F$ is a filter, $G_I$ is an ideal, and they are disjoint. These things are easy to check, and we here only deal with disjointness for illustration. Suppose, to get a contradiction, that $c(u)\geq z\geq r\wedge c(t)$ where $t\in F$ and $u\not\in F$. Then by axioms (7) and (9) we get
\begin{align*}
u&\geq \exists(c(u))\\
&\geq\exists(r\wedge c(t))\\
&=\exists(r)\wedge t\\
&\in F
\end{align*}
putting $u\in F$, a contradiction.

So there is a prime filter $G$ extending $G_F$ and disjoint from $G_I$. This works.
\end{proof}   

Given a prime filter $G$ on sort $n+1$, note that $c^{-1}(G):=\{u\mid c(u)\in G\}$ is always a prime filter on sort $n$ such that $u\in c^{-1}(G)$ if and only if $c(u)\in G$. The above lemma asserts that given any prime filter $F$ on sort $n$ and element $r$ with $\exists(r)\in F$, there is some prime filter $G$ such that $r\in G$ and $c^{-1}(G)=F$. 

Recall that when $L$ is an algebra in the positive existential signature, an almost morphism from $L$ to a positive existential algebra is (essentially) the same thing as a model of the almost morphic conditions. Note that every tuple $(a_1,\ldots,a_n)$ from a model $M$ of the almost morphic conditions gives rise to a prime filter ${\mathfrak p}(\bar{a}):=\{r\mid M\models r(\bar{a})\}$ of $L$ on sort $n$. 

\begin{definition} If $M_1$ and $M_2$ are models of the almost morphic conditions (associated to some algebra $L$ in the positive existential signature), then we say that $M_2$ \textbf{has witnesses over} $M_1$ when $M_1$ is a substructure of $M_2$, written $M_1\subseteq M_2$, and whenever $(a_1,\ldots,a_n)$ is a tuple from $M_1$ and $G$ is a prime filter of $L$ on sort $n+1$ with $c^{-1}(G)={\mathfrak p}(\bar{a})$, then there is some element $b$ in $M_2$ such that $\bar{a}b$ weakly realizes $G$, i.e.\ $M_2\models r(\bar{a}b)$ for every $r\in G$.
Note that I say ``weakly realizes'' instead of ``realizes'' because we do not require that $M_2\models\neg r(\bar{a}b)$ when $r\not\in G$. 
\end{definition}

\begin{lemma}
\label{witnesses}
Let $L$ be an algebra that satisfies axioms (1)-(3), (7)-(10). Let $M_1$ be a model of the almost morphic conditions. Then there is some model $M_2\supseteq M_1$ of the almost morphic conditions which has witnesses over $M_1$. 
\end{lemma}
\begin{proof}
Consider the following first order theory $U$, the signature for which contains a relation symbol for each element of $L$ with arity corresponding to its sort, and also some constants as indicated below:
\begin{enumerate}[(A)]
\item The almost morphic conditions.
\item The literal diagram of $M_1$. I.e.\ for every tuple $\bar{a}$ from $M_1$ and every $r$ in $L$ we write $r(\bar{a})$ when $M_1\models r(\bar{a})$ and $\neg r(\bar{a})$ when $M_1\models\neg r(\bar{a})$ 
\item For each prime filter $G$ in sort $n+1$ of $L$, and each $n$-tuple $\bar{a}$ from $M_1$ with ${\mathfrak p}(\bar{a})=c^{-1}(G)$, we introduce a new constant $y_{G,\bar{a}}$, and then for each $r\in G$, we write
\[r(\bar{a}y_{G,\bar{a}})\]
\end{enumerate}
Items (A) and (B) of the theory ensure that a model satisfies the almost morphic conditions and is a superstructure of $M_1$. Item (C) ensures that a model will have witnesses over $M_1$. By the compactness theorem, we'll thus be done if we can find a model of any given finite amount of items (A), (B), and (C).

Let $U^-$ be consist of finitely many sentences from items (B) and (C).  Only finitely many elements of $M_1$ appear in $U^-$. Collect them together in one big tuple $\bar{a}$, say of length $m$, without duplicates. We will be satisfying all of item (A). Let $(G_1,\bar{a}_1),\ldots,(G_n,\bar{a}_n)$ be the tuple/prime filter pairs that occur in $U^-$ and item (C). We may assume $n\geq 1$ (otherwise we can let $M_2=M_1$). Finitely many of the elements $r\in G_i$ will occur, but we will actually be ensuring things work for all $r\in G_i$, for each $i$. We have $c^{-1}(G_i)={\mathfrak p}(\bar{a}_i)$ where $\alpha_i^{\text{tuple}}(\bar{a})=\bar{a}_i$ for some substitution $\alpha_i$. Let $k_i$ denote the length of $\bar{a}_i$. Define substitutions $\beta_0^{\text{tuple}}(\bar{a}y_1\cdots y_n)=\bar{a}$ and $\beta_i^{\text{tuple}}(\bar{a}\bar{y})=\bar{a}_iy_i$ for $1\leq i\leq n$.

Now I claim that there is a prime filter $H$ on sort $m+n$ such that 
\begin{enumerate}[(I)]
\item For all $r$ in sort $m$ we have $\beta_0(r)\in H$ if and only if $r\in {\mathfrak p}(\bar{a})$, and
\item For each $i=1,\ldots,n$, for each $r$ in sort $k_i+1$ we have $r\in G_i$ implies $\beta_i(r)\in H$.
\end{enumerate} 
Suppose for now that there is such a prime filter. Then we define a model $M_2^-$ with underlying set $\{H_1,\ldots,H_{m+n}\}$ as follows:
\[M_2^-\models r(\gamma(\bar{H}))\iff \gamma(r)\in H\]
We interpret the constants $a_1,\ldots,a_m$ by $H_1,\ldots,H_m$, and the constants $y_{G_1,\bar{a}_1},\ldots,y_{G_n,\bar{a}_n}$ by $H_{m+1},\ldots,H_{m+n}$. 
Of course item (A) is satisfied by Lemma~\ref{almost morphism}. 
Now consider the sentences in $U^-$ and item (C). Let $r\in G_i$. We want $M_2^-\models r(\bar{a}_iy_{G_i,\bar{a}_i})$, i.e.\ $\beta_i(r)\in H$ (because $\beta_i^{\text{tuple}}(\bar{a}\bar{y})=\bar{a}_iy_i$). But this is implied by $r\in G_i$, according to (II).

Finally consider the sentences in $U^-$ and item (B). We show that for any tuple $\gamma^{\text{tuple}}(\bar{a})$ assembled from $\bar{a}$, and for any $r$ of the appropriate sort, we have $M_2^-\models r(\gamma^{\text{tuple}}(\bar{a}))$ if and only if $M_1\models r(\gamma^{\text{tuple}}(\bar{a}))$.
Note that \[\gamma^{\text{tuple}}(\bar{a})=\gamma^{\text{tuple}}(\beta_0(\bar{a}\bar{y}))=(\beta_0\circ\gamma)^{\text{tuple}}(\bar{a}\bar{y})\]
and so by the definition of $M_2^-$, we have $M_2^-\models r(\gamma^{\text{tuple}}(\bar{a}))$ if and only if $\beta_0(\gamma(r))=(\beta_0\circ\gamma)(r)\in H$. By (I), this is equivalent to $\gamma(r)\in {\mathfrak p}(\bar{a})$, i.e.\ $M_1\models(\gamma(r))(\bar{a})$. Since $M_1$ itself satisfies the almost morphic conditions, this is equivalent to $M_1\models r(\gamma^{\text{tuple}}(\bar{a}))$, as desired.
 
Now we show that we can get such an $H$. We use an argument similar to that of Lemma~\ref{sum} or Lemma~\ref{project}. Let $A$ be the distributive lattice which is sort $m+n$ of $L$. Define
\[H_F:=\{z\in A\mid z\geq \beta_0(r)\wedge\bigwedge_{i=1}^n\beta_i(r_i)\text{ for some }r\in {\mathfrak p}(\bar{a})\text{ and }r_i\in G_i\}\]
and 
\[H_I:=\{z\in A\mid \beta_0(r)\geq z\text{ for some }r\not\in {\mathfrak p}(\bar{a})\}\]
Then $H_F$ is a filter, $H_I$ is an ideal, and they are disjoint. The main thing to check is the disjointness. Suppose, to get a contradiction, that 
\[\beta_0(s)\geq z\geq\beta_0(r)\wedge\bigwedge_{i=1}^n\beta_i(r_i)\]
where $s\not\in {\mathfrak p}(\bar{a})$, $r\in {\mathfrak p}(\bar{a})$, and $r_i\in G_i$ for each $i$. Observe that $c^{(n)}=\beta_0$, and so by repeated use of the facts that $\exists(c(t))\leq t$ and $\exists$ is increasing, we get
\[s\geq \exists^{(n)}(\beta_0(s))\geq\exists^{(n)}(\beta_0(r)\wedge\bigwedge_{i=1}^n\beta_i(r_i))\]
By repeated use of axiom (9), the right hand side becomes
\[r\wedge\exists^{(n)}(\bigwedge_{i=1}^n\beta_i(r_i))\]
Putting this all together with axiom (10), we see that
\[s\geq r\wedge\bigwedge_{i=1}^n\alpha_i(\exists(r_i))\]
To get that $s\in {\mathfrak p}(\bar{a})$, a contradiction, we will show that $\alpha_i(\exists(r_i))\in {\mathfrak p}(\bar{a})$ for each $i$. By axiom (8), $c(\exists(r_i))\geq r_i\in G_i$, so $\exists(r_i)\in {\mathfrak p}(\bar{a}_i)$ (recall that $c^{-1}(G_i)={\mathfrak p}(\bar{a}_i)$ by assumption). Then, as $M_1$ satisfies the almost morphic conditions, and $\alpha_i^{\text{tuple}}(\bar{a})=\bar{a}_i$, we get that $\alpha_i(\exists(r_i))\in {\mathfrak p}(\bar{a})$.

A prime filter $H$ which extends $H_F$ and is disjoint from $H_I$ is as desired. 
\end{proof}

If $f\colon A\to B$ is a function, we use $\operatorname{ker}(f)$ to denote the relation $\{(a,a')\in A^2\mid f(a)=f(a')\}$. 

\begin{lemma}
\label{existential}
Let $L$ be an algebra that satisfies axioms (1)-(3) and (7)-(10). Let $\varphi$ be an almost morphism from $L$ to some positive existential algebra. Then there is a morphism $\varphi^+$ from $L$ to some positive existential algebra such that $\ker(\varphi^+)\subseteq\ker(\varphi)$. In particular, if $\varphi$ is 1-1, then $\varphi^+$ is an embedding.  
\end{lemma}
\begin{proof}
The almost morphism $\varphi$ gives rise to a model $M_1$ which satisfies the almost morphic conditions. By Lemma~\ref{witnesses} there is a model $M_2\supseteq M_1$ of the almost morphic conditions which has witnesses over $M_1$. Continuing in this way, we get a sequence
\[M_1\subseteq M_2\subseteq M_3\subseteq\cdots\]
of length $\omega$ where $M_{n+1}$ has witnesses over $M_n$. Let $M^+$ be the union of this chain of models. Since the almost morphic conditions are of a form preserved by unions of chains, we get that $M^+$ models them too. 

Further, if $M^+\models (\exists(r))(\bar{a})$, then $M_n\models(\exists(r))(\bar{a})$ for some $n$. By Lemma~\ref{project}, there is a prime filter $G$ such that ${\mathfrak p}(\bar{a})=c^{-1}(G)$ and $r\in G$. Since $M_{n+1}$ has witnesses over $M_n$, there is some element $b\in M_{n+1}$ such that $M_{n+1}\models r(\bar{a}b)$. Thus, $M^+\models r(\bar{a}b)$. In summary, $M^+\models(\exists(r))(\bar{a})$ implies $M^+\models \exists y(r(\bar{a}y))$. Thus, the function given by $\varphi^+(r):=r^{M^+}$ is a morphism from $L$ to the positive existential algebra of relations on the underlying set of $M^+$.

Finally, suppose $\varphi^+(r)=\varphi^+(s)$. We show that $\varphi(r)=\varphi(s)$. If there were $\bar{a}$ from $M_1$ with $\bar{a}\in\varphi(r)-\varphi(s)$, then $\bar{a}\in\varphi^+(r)-\varphi^+(s)$ as well, because $M_1\subseteq M^+$.
\end{proof}

From Lemma~\ref{Horn separation} and Lemma~\ref{existential} we get the following proposition:

\begin{proposition}[B\"orner]
\label{Horn for positive existential}
Axioms (1)-(4) and (7)-(10) equationally axiomatize the subalgebras of products of the positive existential algebras.
\end{proposition}

Theorem~\ref{main theorem positive existential} follows immediately from Lemma~\ref{1-1 almost morphism} and Lemma~\ref{existential}.

\subsection{Equality}\label{equality}

If we wish to add equality, we may do so (modularly) with the following axioms. In this section we do not provide a detailed analysis, but rather just indicate briefly how the above argument changes. 
\begin{center}
\textbf{Axioms for Equality}
\end{center}
\begin{enumerate}[(1)]
\setcounter{enumi}{10}
\item
\begin{enumerate}
\item $\Delta_{i,i}^n=1$
\item $\Delta_{i,j}^n=\Delta_{j,i}^n$
\item $\Delta_{i,j}^n\wedge\Delta_{j,k}^n\leq\Delta_{i,k}^n$
\end{enumerate}
\item When $\alpha,\beta\colon k\to n$ are substitutions of matching arities we have the axiom:
\[\alpha(r)\wedge\bigwedge_{l=1}^k \Delta_{\alpha(l),\beta(l)}^n=\beta(r)\wedge\bigwedge_{l=1}^k\Delta_{\alpha(l),\beta(l)}^n\]
\item For each substitution $\alpha\colon k\to n$ we have the axiom:
\[\alpha(\Delta_{i,j}^k)=\Delta_{\alpha(i),\alpha(j)}^n\]
\end{enumerate}

It is straightforward to check that these equational axioms are all true in the concrete algebras where $\Delta_{i,j}^n$ is interpreted as the $n$-ary relation which holds of an $n$-tuple if and only if the $i^{th}$ and $j^{th}$ coordinates are equal. Axiom (11) corresponds to the usual properties of an equivalence relation. Axiom (12) is algebraically saying the obvious fact that 
\begin{align*}
\{\bar{x}\mid \alpha^{\text{tuple}}(\bar{x})\in r&\text{ and }\alpha^{\text{tuple}}(\bar{x})=\beta^{\text{tuple}}(\bar{x})\}\\
&=\{\bar{x}\mid \beta^{\text{tuple}}(\bar{x})\in r\text{ and }\alpha^{\text{tuple}}(\bar{x})=\beta^{\text{tuple}}(\bar{x})\}
\end{align*}
Finally, to make sense of axiom (13), recall that the $i^{th}$ coordinate of $\alpha^{\text{tuple}}(\bar{x})$ is $x_{\alpha(i)}$. So, $\alpha^{\text{tuple}}(\bar{x})\in\Delta_{i,j}^k$ if and only if $x_{\alpha(i)}=x_{\alpha(j)}$. 

Let $F$ be a prime filter on sort $n$. Our basic strategy is the same --- get a modified version of Lemma~\ref{morphism} by having $F$ correspond to a tuple $(F_1,\ldots,F_n)$ --- except that now we may have to identify certain of the $F_i$. By axiom (11), we may put
\[F_i=F_j\iff\Delta_{i,j}^n\in F\]
That is, the relation $\{(i,j)\mid \Delta_{i,j}\in F\}$ is an equivalence relation. In detail, it is reflexive by axiom (11) part (a). It is symmetric by axiom (11) part (b). And it is transitive by axiom (11) part (c).
But then we may have $\alpha^{\text{tuple}}(\bar{F})=\beta^{\text{tuple}}(\bar{F})$ for distinct substitutions $\alpha$ and $\beta$. 
The upshot of axiom (12) is that this won't matter: 
If $\alpha^{\text{tuple}}(\bar{F})=\beta^{\text{tuple}}(\bar{F})$ then $\alpha(r)\in F\iff\beta(r)\in F$. To see this, observe that
\begin{align*}
\alpha^{\text{tuple}}(\bar{F})=\beta^{\text{tuple}}(\bar{F})&\iff F_{\alpha(l)}=F_{\beta(l)}\text{ for each }l=1,\ldots,k\\
&\iff \bigwedge_{l=1}^k\Delta_{\alpha(l),\beta(l)}^n\in F
\end{align*}
So if $\alpha(r)\in F$ and $\alpha^{\text{tuple}}(\bar{F})=\beta^{\text{tuple}}(\bar{F})$, then we get
\begin{align*}
\beta(r)&\geq\beta(r)\wedge\bigwedge_{l=1}^k \Delta_{\alpha(l),\beta(l)}^n\\
&=\alpha(r)\wedge\bigwedge_{l=1}^k\Delta_{\alpha(l),\beta(l)}^n\\
&\in F
\end{align*}
and so $\beta(r)\in F$.

Now we are in a position to obtain the with-equality version of Lemma~\ref{morphism}, using axiom (13) for the preservation of the $\Delta_{i,j}^n$. For definiteness, we state the lemma for positive quantifier-free algebras with equality.

\begin{lemma}
\label{morphism with equality}
Let $L$ be an algebra that satisfies axioms (1)-(3), (11)-(13). Let $F$ be a prime filter on sort $n$. Let $F_1,\ldots,F_n$ be symbols such that $F_i=F_j$ if and only if $\Delta_{i,j}^n\in F$. Let $W=\{F_1,\ldots,F_n\}$. Let $A(W)$ be the positive quantifier-free algebra with equality on the relations of $W$. Then $\varphi\colon L\to A(W)$ defined by \[\alpha^{\operatorname{tuple}}(\bar{F})\in\varphi(r)\iff\alpha(r)\in F\]
is a morphism.
\end{lemma}
\begin{proof}
As observed above, this definition of $\varphi$ is unambiguous by axiom (12).
 
The new thing we have to check is that \[\varphi(\Delta_{i,j}^k)=\{\alpha^{\text{tuple}}(\bar{F})\mid F_{\alpha(i)}=F_{\alpha(j)}\}\]
Well,
\begin{align*}
\varphi(\Delta_{i,j}^k)&=\{\alpha^{\text{tuple}}(\bar{F})\mid \alpha(\Delta_{i,j}^k)\in F\}\\
&=\{\alpha^{\text{tuple}}(\bar{F})\mid \Delta_{\alpha(i),\alpha(j)}^n\in F\}\\
&=\{\alpha^{\text{tuple}}(\bar{F})\mid F_{\alpha(i)}=F_{\alpha(j)}\}
\end{align*}
\end{proof}

\section{Theories}
\label{theories}

We now consider how formulas and theories may be understood in the context of our multisorted algebraic approach. As mentioned in the introduction, thinking about first order theories in an algebraic way is not new, and it was discussed even in our multisorted formalism by B\"orner in Section~3.7 of \cite{Boerner}. However, our discussion here will help explain the value of axiom (0) in letting us have a uniform argument for the various reducts.

Let us say a first order \textbf{formula} in some relational signature $\sigma$ is an element of the free algebra (in the first order algebra signature) generated by the symbols of $\sigma$ (which are relation symbols of various fixed finite arities). Let us use the notation $F_\sigma$ to refer to this free algebra. Positive existential formulas, quantifier-free formulas, etc. are defined correspondingly. For an example, let $\sigma$ consist of a unary relation symbol $R$ and a binary relation symbol $S$. Let $\alpha\colon 2\to 2$ be the substitution $\alpha^{\text{tuple}}(x,y)=(y,x)$. Then 
\[R,\,\, S,\,\, \alpha(S),\,\, \exists(\alpha(S)),\,\, R\wedge\exists(\alpha(S))\]
are some formulas.

This way of viewing formulas does away with bound/free variables and the associated ``alphabetic variants", but of course a formula up to logical equivalence may have more than one syntactic representation in this formalism as well (e.g.\ $\alpha(\alpha(S))$ and $S$ are logically equivalent). Also note that the variable context has now become an intrinsic part of the formula (its arity). 

A first order $\sigma$-\textbf{structure} is a morphism from $F_\sigma$ to some first order algebra $M$. This is the same as a function which assigns to every relation symbol of $\sigma$ a relation on the underlying set of $M$. Let us use $K$ to denote the class of concrete algebras for the kind of logic under consideration (i.e.\ $K$ could be the first order algebras, or the positive existential algebras, etc.). Then a $\sigma$-structure for whatever logic is under consideration is a morphism from $F_\sigma$ to an algebra $M\in K$. 

Given any collection $T$ of identities of formulas (i.e.\ pairs of formulas from the same sort), the statement that a structure $f\colon F_\sigma\to M$ is a \textbf{model} of $T$ means that $f(r)=f(s)$ for each pair $(r,s)\in T$. A (partial) \textbf{theory} $T$ is an ``implicationally closed" collection of identities in the sense that if every model $f$ of $T$ satisfies $f(r)=f(s)$, then also $(r,s)\in T$. Every theory is in particular a congruence relation on $F_\sigma$. Thus, we have an associated quotient $F_\sigma/T$, which could be called the theory too. A morphism $F_\sigma/T\to M\in K$ is the same thing as a model of $T$. The algebras that arise as quotients in this way (i.e., are of the form $F_\sigma/T$ for some signature $\sigma$ and some $\sigma$-theory $T$) are exactly the subalgebras of the products of the concrete algebras, i.e.\ $SP(K)$. 

We include the easy verification of this fact for illustrative purposes. First let $Q=F_\sigma/T$ be such a quotient. We now prove that $Q$ must satisfy the equational theory of $K$ (which is equivalent to the Horn clause theory for the $K$ of present interest). Let $\varphi(\bar{r})=\chi(\bar{r})$ be an equation true in all members of $K$. Let $\bar{r}$ be some tuple from $Q$. Then let $f\colon Q\to M\in K$ be any model of $T$. Of course we must have $\varphi(f\bar{r})=\chi(f\bar{r})$. Since $f$ is a morphism, this yields $f(\varphi(\bar{r}))=f(\chi(\bar{r}))$. This works for any model $f$, and so by the assumption that $T$ is implicationally closed, we get that $Q\models\varphi(\bar{r})=\chi(\bar{r})$ too. Since $Q$ satisfies the equational theory of $K$, by Proposition~\ref{Horn clause} or its analogue, we get that $Q\in SP(K)$.  

Conversely, let $Q\in SP(K)$. Specifically let $Q\subseteq\prod_{i\in I} M_i$ where the $M_i$ are in $K$. Introduce a signature $\sigma$ with a symbol for each element of $Q$. Then $F_\sigma/T=Q$ for some congruence $T$. We claim $T$ is a theory, i.e.\ is implicationally closed. Suppose $r,s\in Q$ with $f(r)=f(s)$ for all morphisms $f\colon Q\to M\in K$. Then in particular for the projections $\pi_i\colon Q\to M_i$ ($i\in I$) we have $\pi_i(r)=\pi_i(s)$. So $r=s$.  

We may thus say that theories are simply subalgebras of products of the concrete algebras in question (with specified generators). The usual notion of first order theory is an (implicationally closed) collection of sentences (identities of the form $\varphi=1$ in sort zero). In the first order case, where universal quantification and the biconditional are present, this agrees with the notion of theory described above, essentially because $r=s$ in a first order algebra if and only if $\forall^{(n)}(r\leftrightarrow s)=1$ (where $r$ and $s$ are in sort $n$). Intuitively speaking, in the first order signature, the zero sort controls all the sorts. For the reducts this is not true. To illustrate this point, and to help explain the value of axiom (0), we now show that first order theories with exactly two elements in sort zero are the ones in $S(K)$, but importantly that this characterization does not hold for the reducts.

When considered as a collection of sentences, a first order theory is said to be \textbf{complete} when every sentence or its negation (but not both) is in the theory. Translating this to the quotient view of theories, this says that there are exactly two elements in sort zero. Of course any subalgebra of a first order algebra is going to be a theory with exactly two elements of sort zero. But the converse is true as well, \textit{when negation and projection are present}. We check that axiom (0) follows from the Horn clause theory of first order algebras together with the assumption that there are exactly two elements of sort zero. 

First we observe that in any algebra satisfying the Horn clause theory of first order algebras, for any element $t$ of sort $k$ we have $t=0\iff \exists^{(k)}(t)=0$, because the two directions of this bi-implication are both Horn clauses true of first order algebras (let us say are ``true Horn clauses"). If additionally we have an algebra with exactly two elements in sort zero (0 and 1), then $t\neq 0$ if and only if $\exists^{(k)}(t)=1$. 

We now prove axiom (0) in contrapositive form. Let $s_i\not\geq r_i$ for each $i=1,\ldots,m$, where $r_i,s_i\colon k_i$, and let $c_i\colon k_i\to (k_1+\cdots+k_m)=n$ be partitioning cylindrifications. Since $t\wedge\neg u=0\implies u\geq t$ is a true Horn clause, we get that $r_i\wedge\neg s_i\neq 0$ for each $i$. Thus, $\exists^{(k_i)}(r_i\wedge\neg s_i)=1$. Another true Horn clause is 
\[\bigwedge_{i=1}^m \exists^{(k_i)}(t_i)=1\implies\exists^{(n)}\bigwedge_{i=1}^m c_i(t_i)=1\] Thus, we get in our case
\[\exists^{(n)}\bigwedge_{i=1}^m c_i(r_i\wedge\neg s_i)=1\]
So 
\[\bigwedge_{i=1}^m c_i(r_i\wedge\neg s_i)\neq 0\]
which simplifies to $\dd\bigwedge_{i=1}^m c_i(r_i)\wedge\neg\bigvee_{i=1}^m c_i(s_i)\neq 0$. So $\dd\bigwedge_{i=1}^m c_i(r_i)\not\leq\bigvee_{i=1}^m c_i(s_i)$. 

So, we could have presented an axiomatization of the universal theory of first order algebras by just taking the Horn clause theory and adding to it the axiom that there are exactly two elements in sort zero. However, this would not have yielded results uniformly for the reducts as well. There is a model of the Horn clause theory of positive existential algebras which has exactly two elements of sort zero, but fails to satisfy axiom (0). To see this, consider the (partial) first order theory (presently we will be taking a positive existential reduct) in a language with three unary relation symbols $R$, $A$, and $B$ generated by the following sentences:
\begin{enumerate}[(i)]
\item $\exists x(A(x)\wedge B(x))$
\item $\forall x(R(x)\iff A(x))\vee\forall x(R(x)\iff B(x))$
\end{enumerate}  
Let $Q$ be the associated subalgebra of a product of first order algebras. Consider the positive existential reduct of $Q$, and then consider the subalgebra generated by $R$, $A$, and $B$. Call it $Q_0$, and note that $Q_0$ is itself a subalgebra of a product of positive existential algebras. Note that $X\in Q_0$ if and only if there is some positive existential formula $\varphi$ such that $\varphi(R,A,B)=X$. Because we're dealing with unary relation symbols, and projections of conjunctions of some of $R,A,B$ are predictably 1, we in fact may assume that $\varphi$ is positive quantifier-free. Every element of sort zero in $Q_0$ is obtained by projecting an element of sort one. One can check the only possible values are 0 and 1 (and $0\neq 1$ because our theory has a model). On the other hand, letting $c_1^{\operatorname{tuple}}(x,y)=x$ and $c_2^{\operatorname{tuple}}(x,y)=y$, we have $c_1(A)\wedge c_2(B)\leq c_1(R)\vee c_2(R)$ (i.e.\ $A(x)\wedge B(y)\models R(x)\vee R(y)$), but $A\not\leq R$ and $B\not\leq R$, violating axiom (0). 

It is easy to give a theory in the quantifier-free signature which does not satisfy axiom (0) and still has exactly two elements in sort zero, because there are no functions going from the higher sorts to sort zero in this case. So any violation of axiom (0) not involving sort zero yields an example. For instance, consider the quantifier-free algebra $A(W)$ of relations on a set $W$ of one element. Then the product $L:=A(W)\times A(W)$ has a ``diamond" for each sort. Let us use 0, $a$, $b$, and 1 to denote the elements of $L$ in sort one. Let $c_1^{\operatorname{tuple}}(x,y)=x$ and $c_2^{\operatorname{tuple}}(x,y)=y$ be partitioning cylindrifications. Then $c_1(a)\wedge c_2(b)$ is the bottom element in sort two. Thus, $c_1(b)\vee c_2(0)\geq c_1(a)\wedge c_2(b)$. However, $b\not\geq a$ and $0\not\geq b$, violating axiom (0).

\section{Dealing with Function Symbols}\label{function symbols}

We briefly indicate how to deal with function symbols. Let $\pi$ be a fixed functional signature. We have terms $\alpha(x_1,\ldots,x_n)$ defined as usual (elements of the free $\pi(\bar{x})$-algebra where the $\bar{x}$ are extra constant symbols). From these we obtain ``term-tuples"
\[\alpha(\bar{x})=(\alpha_1(\bar{x}),\ldots,\alpha_k(\bar{x}))\]
where each $\alpha_i(\bar{x})$ is a term.

The term-tuples induce operations on tuples of a $\pi$-algebra $W$ in the obvious way. Given a tuple $\bar{x}\in W^n$, we get $\alpha(\bar{x})=(\alpha_1(\bar{x}),\ldots,\alpha_k(\bar{x}))\in W^k$. The inverse images of these are operations on relations going in the reverse direction $\alpha\colon {\mathcal P}(W^k)\to {\mathcal P}(W^n)$. The substitutions are obtained as a special case for any signature $\pi$, and when $\pi$ is the empty signature, they are the only term-tuples. The multisorted signature of interest to us now has an operation of arity $\alpha\colon k\to n$ for each such term-tuple, and the concrete algebras of interest are the ones that arise from considering the relations on a $\pi$-algebra.

With respect to axiomatization, if equality is not present, we need only change axioms (2), (3), and (5) by expanding their scope to include all term-tuples (not just substitutions). 

As for equality, let us have a constant $\Delta_{\alpha,\beta}$ of sort $n$ for each pair of term-tuples $\alpha,\beta\colon k\to n$. The intended interpretation is $\Delta_{\alpha,\beta}:=\{\bar{x}\mid \alpha(\bar{x})=\beta(\bar{x})\}$. In the special case $\alpha(\bar{x})=x_i$ and $\beta(\bar{x})=x_j$, we get $\Delta_{\alpha,\beta}=\Delta_{i,j}^n$. 

Then we rewrite axioms (11)-(13) as follows:
\renewcommand{\labelitemii}{$\cdot$}
\begin{itemize}
 \item \begin{itemize}
\item $\Delta_{\alpha,\alpha}=1$
\item $\Delta_{\alpha,\beta}=\Delta_{\beta,\alpha}$
\item $\Delta_{\alpha,\beta}\wedge\Delta_{\beta,\gamma}\leq\Delta_{\alpha,\gamma}$
\end{itemize}
\item $\alpha(r)\wedge\Delta_{\alpha,\beta}=\beta(r)\wedge\Delta_{\alpha,\beta}$
\item $\gamma(\Delta_{\alpha,\beta})=\Delta_{\gamma\circ\alpha,\gamma\circ\beta}$
\end{itemize}
To these we also add
\begin{itemize}
 \item $\Delta_{\alpha,\beta}=\bigwedge_{i=1}^k\Delta_{\alpha_i,\beta_i}$ where $\alpha=(\alpha_1,\ldots,\alpha_k)$ and $\beta=(\beta_1,\ldots,\beta_k)$
\item $\Delta_{\alpha,\beta}\leq\Delta_{\alpha\circ\gamma,\beta\circ\gamma}$
\end{itemize}

So how do the proofs get modified? The only essential change is with the analogues of Lemma~\ref{morphism}. We want to have a prime filter $F$ on sort $n$ of an abstract algebra satisfying the axioms give rise to a morphism to a concrete algebra. When equality isn't present, instead of letting $W=\{F_1,\ldots,F_n\}$, we let $W$ be the free $\pi$-algebra with $F_1,\ldots,F_n$ as generators. In the special case where we have no function symbols, i.e.\ $\pi$ is empty, we get back the old $W$. We define the morphism as before: $\alpha(\bar{F})\in\varphi(r)\iff \alpha(r)\in F$, except that now $\alpha$ may range over all the term-tuples, not just the substitutions.

When equality is present, we additionally identify certain elements of this free algebra by saying $\alpha(\bar{F})=\beta(\bar{F})$ if and only if $\Delta_{\alpha,\beta}\in F$. The additional axioms ensure that this makes sense and in fact gives a congruence relation.

Finally, the fixed functional signature $\pi$ can also be taken to be multisorted, with only minor modifications to our argument.

\section{The Axioms}

For ease of reference, here is a list of the main axioms considered:
\begin{enumerate}[(1)]
\setcounter{enumi}{-1}
\item  When $c_1,\ldots,c_m$ are partitioning cylindrifications we have the axiom: 
\begin{center} 
 If $\dd\bigvee_{i=1}^m c_i(s_i)\geq\bigwedge_{i=1}^m c_i(r_i)$, then $s_i\geq r_i$ for some $i=1,\ldots,m$.    
\end{center}
\item $0,1,\vee,\wedge$ form a (bounded) distributive lattice in each sort. 
\item Substitutions preserve $0,1,\vee,\wedge$
\item $(\beta\circ\alpha)(r)=\beta(\alpha(r))$
\item $\operatorname{id}(r)=r$ 
\item $\alpha(\neg r)=\neg\alpha(r)$
\item $r\vee\neg r=1$, $r\wedge\neg r=0$
\item $\exists$ preserves 0 and $\vee$
\item $r\leq c(\exists(r))$
\item $\exists(r\wedge c(s))=\exists(r)\wedge s$ 
\item Let $\alpha_i^{\text{tuple}}(\bar{x})$ be substitutions for $i=1,\ldots,n$. Define $\beta_i^{\text{tuple}}(\bar{x}y_1\cdots y_n)=\alpha_i^{\text{tuple}}(\bar{x})y_i$. Then we have the axiom \[\exists^{(n)}(\bigwedge_{i=1}^n\beta_i(r_i))=\bigwedge_{i=1}^n\alpha_i(\exists(r_i))\]
where $\exists^{(n)}$ means we apply projection $n$ times.
\item
\begin{enumerate}
\item $\Delta_{i,i}^n=1$
\item $\Delta_{i,j}^n=\Delta_{j,i}^n$
\item $\Delta_{i,j}^n\wedge\Delta_{j,k}^n\leq\Delta_{i,k}^n$
\end{enumerate}
\item When $\alpha,\beta\colon k\to n$ are substitutions of matching arities we have the axiom:
\[\alpha(r)\wedge\bigwedge_{l=1}^k \Delta_{\alpha(l),\beta(l)}^n=\beta(r)\wedge\bigwedge_{l=1}^k\Delta_{\alpha(l),\beta(l)}^n\]
\item For each substitution $\alpha\colon k\to n$ we have the axiom:
\[\alpha(\Delta_{i,j}^k)=\Delta_{\alpha(i),\alpha(j)}^n\]
\end{enumerate}

\end{document}